\documentclass[text=11pt]{article}

\usepackage{amsfonts,amsmath,amssymb,amsthm,commath,wasysym}
\usepackage[utf8]{inputenc}
\usepackage[english]{babel}
\usepackage{bbm}
\usepackage{enumitem}
\usepackage{color}
\usepackage{comment}
\usepackage[margin=1.02in]{geometry}
\usepackage{setspace}
\usepackage{hyperref}
\usepackage[sort]{cite}

\newtheorem{theorem}{Theorem}

\newtheorem{proposition}[theorem]{Proposition}
\newtheorem{corollary}[theorem]{Corollary}

\newtheorem{lemma}[theorem]{Lemma}
\theoremstyle{remark}
\newtheorem{remark}{Remark}

\DeclareMathOperator*{\esssup}{ess\,sup}

\numberwithin{theorem}{section}

\DeclareMathOperator{\Z}{\mathbb{Z}}
\DeclareMathOperator{\R}{\mathbb{R}}

\DeclareMathOperator{\N}{\mathbb{N}}

\newcommand{\E}[1]{\mathbb{E}\left[#1\right]}
\newcommand{\eps}{\varepsilon}
\newcommand{\Em}[2]{\mathbb{E}^{#1}\left[#2\right]}

\DeclareMathOperator{\pr}{\mathbb{P}}

\newcommand*\diff{\mathop{}\!\mathrm{d}}

\newcommand{\Ex}[1]{\mathbb{E}\left[#1\right]}

\newcommand{\ind}[1]{\mathbbm{1}_{\{#1\}}}

\newcommand{\email}[1]{\href{mailto:#1}{#1}}

\title{Collisions of Random  Walks in Dynamic Random Environments}
\author{Noah Halberstam and Tom Hutchcroft}

\begin{document}

\maketitle

\begin{abstract}
We study dynamic random conductance models on $\Z^2$ in which the environment evolves as a reversible Markov process that is stationary under space-time shifts. We prove under a second moment assumption that two conditionally independent random walks in the same environment collide infinitely often almost surely. These results apply in particular to random walks on dynamical percolation. 

\end{abstract}

\section{Introduction}

A graph is said two have the \textbf{infinite collisions property} if two independent random walks started at the same location collide (occupy the same location at the same time) infinitely often almost surely.  For Euclidean lattices, Polya observed that the study of collisions can be reduced to the study of returns on an auxiliary lattice, and hence that the infinite collisions property holds if and only if the dimension is at most two.
In fact, \emph{for transitive graphs}, the infinite collisions property is always equivalent to recurrence: The number of collisions and the number of returns are geometric random variables with the same mean. 
For bounded degree graphs that are \emph{not} transitive, the infinite collisions property is strictly stronger than recurrence. Indeed, while it is easy to see that bounded degree transient graphs cannot have infinite collisions, Krishnapur and Peres \cite{krishnapur2004} showed that there exist bounded degree graphs, including the \textit{infinite comb graph}, that are recurrent but which do not have the infinite collisions property. See e.g.\ \cite{chen2011} for further examples.

Despite the existence of these counterexamples, it is natural to expect that the infinite collisions property is equivalent to recurrence for most graphs and networks arising in applications. Indeed, it is now known that the two properties are equivalent for many random walks in random environments that are spatially homogeneous in some distributional sense \cite{CCC,barlow2012}. The most general such result is due to Hutchcroft and Peres \cite{hutchcroft2015}, who proved that every recurrent \textit{reversible random rooted network} has the infinite collisions property. An important class of examples to which these result apply are the translation-invariant random conductance models on $\Z^d$; see \cite{biskup2011recent} for background. Note that while earier results such as those of \cite{barlow2012} had relied on a fine analysis of the random walk in specific examples, the method of \cite{hutchcroft2015} is entirely qualitative and does not rely on heat-kernel estimates. Further results on collisions of random walks in random environments include \cite{gallesco_2013,gantert2014recurrence,devulder2018collisions,devulder2019,chen2016gaussian}.

In this paper we study collisions of random walks on \textit{dynamic} random conductance models (dynamic RCMs), in which the environment itself is permitted to vary over time. Such models have recently been of burgeoning interest, with works establishing, for example, quenched invariance principles \cite{andres2014,BISKUP2018985,andres2018}, quenched and annealed local limit theorems \cite{andres2019local,ACS}, heat kernel estimates \cite{NashDyn,SDE}, and Green kernel asymptotics \cite{andres2020green}.
We restrict attention to the class of dynamic RCMs in which the conductances themselves form a \emph{strongly reversible} Markov process whose law is invariant under space-time shifts. We will refer to such environments as \textbf{stationary, strongly reversible Markovian environments}; see Section \ref{Background} for detailed definitions.
This class includes many of the most natural and interesting examples of dynamic RCMs appearing in the literature, including dynamical percolation \cite{Peres2015,Peres2017MixingTF,Hermon2020ACP,Peres2017MixingTF,peres2017quenched}, the simple symmetric exclusion process \cite{SPITZER1970246,avena2012,Redig2018SymmetricSE}, and  dynamic RCMs in which the conductances evolve according to an SDE such as those arising in the Helffer-Sjöstrand representation of gradient fields, see e.g.\ \cite{Helffer1994OnTC,SDE}.
Previous works studying random walks in general (reversible and non-reversible) Markovian environments include \cite{RWME,AVENA20183490,avena2016class}.

We now state our main theorem. We write $E_d$ for the edge set of $\Z^d$, and consider our random environments to be random locally integrable functions from $\R\times E_d$ to $[0,\infty)$. We say that a stationary Markovian random environment $\eta : \R\times E_d \to [0,\infty)$ is \textbf{strongly reversible} if the conditional distributions of $\eta$ and its reversal given the instantaneous sigma-algebra $\mathcal{F}_0$ are almost surely equal, where $\mathcal{F}_{[s,t]}$ is the sigma-algebra generated by the restriction of $\eta$ to $[s,t]$ and $\mathcal{F}_0 := \bigcap \{\mathcal{F}_{[s,t]} : s \leq 0 \leq t, s<t\}$; see Section~\ref{Background} for more detailed definitions.

\begin{theorem} \label{thm:main}
	Let $\eta:\R\times E_2\rightarrow[0,\infty)$ be a stationary random environment on $\Z^2$ and let $(X_t)_{t\in \R}$ and $(Y_t)_{t\in \R}$ be two doubly-infinite random walks on $\eta,$ both started from the origin at time zero, that are conditionally independent given the environment $\eta.$ Suppose that at least one of the following conditions holds:
	\begin{enumerate}[leftmargin=1.1cm]
		\item[\emph{(A1)}:] The environment $\eta$ is Markovian, strongly reversible, and satisfies the second moment condition $\norm{\eta}_2^2:=\sup_{a<b} \frac{1}{|b-a|^2} 
		\mathbb{E}[(\int_{a}^b \sum_{x\sim0}\eta_s(\{0,x\}) \dif s)^2]<\infty$.
		\item[\emph{(A2)}:] The backwards walk $(X_{-t})_{t\geq 0}$ satisfies a (quenched or annealed) invariance principle under Brownian scaling with Brownian motion on $\R^2$ as the limiting distribution.
	\end{enumerate}
	Then $X$ and $Y$ collide infinitely often almost surely:  the set $\{n\in\N:X_n=Y_n\}$ has infinite cardinality almost surely and the set $\{t\in[0,\infty):X_t=Y_t\}$ has infinite Lebesgue measure almost surely.
\end{theorem}

Invariance principles are known in the ergodic setting in the non-elliptic case with rates bounded from above (and $0$ only on intervals with lengths of finite expectation) \cite{BISKUP2018985}, and with elliptic rates under moment conditions on the conductances and their reciprocals \cite{andres2019local}. 
 Thus, dynamical percolation and the simple symmetric exclusion process are covered by either hypothesis (A1) or (A2). We stress however that the proof using reversibility is self-contained and does not rely on any previous results on dynamical percolation.

Both results will be deduced from the following more general theorem. 
Note that the hypotheses of this theorem hold trivially under the assumption (A2) of Theorem~\ref{thm:main}; In Section~\ref{subsec:Markovtype} we use the theory of \emph{Markov-type inequalities} to prove that they also hold under the assumption (A1).

\begin{theorem}[A weak diffusive estimate suffices]
\label{thm:main_diffusive}
	Let $\eta:\R\times E_2\rightarrow\R_{\geq 0}$ be a stationary random environment on $\Z^2$ and let $(X_t)_{t\in \R}$ and $(Y_t)_{t\in \R}$ be two doubly-infinite random walks on $\eta,$ both started from the origin, that are conditionally independent given the environment $\eta.$ Suppose that for every $\eps>0$ there exists $K<\infty$ and $\delta>0$ such that
	\begin{equation}
	\label{eq:diffusive_assumption}
	\mathbb{P}\left(\limsup_{n\to\infty} \min_{0 \leq m \leq n} \mathbb{P}^\eta\Bigl( \|X_{-m}\|_2 \leq K \sqrt{n}\Bigr) \geq \delta \right) \geq 1-\eps.
	\end{equation}
	Then $X$ and $Y$ collide infinitely often almost surely:  the set $\{n\in\N:X_n=Y_n\}$ has infinite cardinality and the set $\{t\in\R_{\geq 0}:X_t=Y_t\}$ has infinite Lebesgue measure almost surely.
\end{theorem}

Under some additional non-degeneracy assumptions, we are able to prove similar infinite-collision theorems in which the two walks $X$ and $Y$ are not required to start at the same location.
 We say a random environment $\eta$ is \textbf{irreducible} if for each two vertices $x$ and $y$ there exist times $s<t$ such that the conditional transition probability $P^\eta_{s,t}(x,y)$ is positive with positive probability. We say that a stationary environment $\eta$ is \textbf{time-ergodic} if it has probability either zero or one to belong to any time-shift-invariant measurable subset of $\Omega$. (Note that being time-ergodic is a stronger condition than being space-time ergodic.)

\begin{corollary} \label{cor:startloc}
Let $\eta:\R\times E_2\rightarrow[0,\infty)$ be a irreducible, time-ergodic, stationary random environment on $\Z^2$ and let $(X_t)_{t\in \R}$ and $(Y_t)_{t\in \R}$ be two doubly-infinite random walks on $\eta$, started at two vertices $x$ and $y$ at time zero, that are conditionally independent given the environment $\eta$. If $\eta$ satisfies the hypotheses of either Theorem~\ref{thm:main} or Theorem~\ref{thm:main_diffusive} then $X$ and $Y$ collide infinitely often almost surely:  the set $\{n\in\N:X_n=Y_n\}$ has infinite cardinality and the set $\{t\in[0,\infty):X_t=Y_t\}$ has infinite Lebesgue measure almost surely.
\end{corollary}

\noindent \textbf{Corollaries for the voter model.} 
Let us now briefly describe a corollary of our results for the \emph{voter model} in two-dimensional dynamic random environments.
Roughly speaking, the voter model in the environment $\eta:\R \times E_d \to \R$ is the interacting particle system on $\Z^d$ in which each vertex has an opinion belonging to some finite set and the opinion of $x$ changes to match the opinion of $y$ at rate $\eta_t(\{x,y\})$. 
 Since this model is tangential to the main results of this paper, we omit the precise definition of the model and refer the reader to \cite{liggett2012interacting} for background.
  The following is an immediate consequence of Corollary~\ref{cor:startloc} and the standard duality between the voter model and coalescing random walk described in \cite[§5]{liggett2012interacting} and \cite[§14]{aldous1995reversible}, which readily generalises to the dynamic case.

\begin{corollary} \label{voter}
	Let $\eta:\R\times E_2\rightarrow\R_{\geq 0}$ be a stationary random environment on $\Z^2$. If the \emph{reversal} of $\eta$ satisfies the hypotheses of Corollary~\ref{cor:startloc}, then the only ergodic stationary measures for the voter model in $\eta$ are the constant 
	(a.k.a.\ consensus) measures. 
\end{corollary}
	
\noindent \textbf{One-dimensional models.}
Our methods can also be used to prove that one-dimensional stationary random environments have the infinite collision property under a \emph{first moment} condition. This is much simpler than the two-dimensional case. Once this proposition is proven, one can also formulate and prove one-dimensional analogues of Corollaries \ref{cor:startloc} and \ref{voter} similarly to the two-dimensional case; we omit the details.

\begin{proposition} \label{prop:1dim}
	Let $\eta$ be a stationary random environment on $\Z$ with $\|\eta\|_1<\infty$. Then $\eta$ has the infinite collisions property almost surely: If $X$ and $Y$ are two random walks in $\eta$ that are conditionally independent given $\eta,$  then the set $\{n\in\N:X_n=Y_n\}$ has infinite cardinality and the set $\{t\in[0,\infty):X_t=Y_t\}$ has infinite Lebesgue measure almost surely.
\end{proposition}


\noindent \textbf{About the proof and organisation.}
This remainder of this paper will be divided into two sections. In Section 2 we define necessary terminology, before establishing moment bounds on the number of jumps the random walk takes in a given interval, as well as non-explosivity in Proposition \ref{prop:ticks}. Then, in Proposition \ref{prop:mTDist}, we use the Markov-Type inequality, along with the previously derived moment bounds, to prove a diffusive upper bound on the displacement of the random walk on the environment.

In Section 3, we will use these results to complete the proof of the theorem. In Proposition \ref{prop:EToP}, we extend to the time-varying setting an argument of Hutchcroft and Peres \cite{hutchcroft2015} to give a sufficient condition for dynamic environments to satisfy the infinite collisions property. Namely, we prove, utilizing the Mass-transport Principle, that if the expected number of collisions \emph{of the backwards walks} conditioned on the environment is infinite almost surely, then the number of collisions is infinite almost surely.
Then, in Theorems \ref{thm:main} and \ref{thm:main_diffusive}, we complete the proof by demonstrating that in two dimensions, the diffusive bound on displacement implies the previously derived sufficient condition on the conditional expectations. We finish by proving Corollary \ref{cor:startloc}.

\section{Stationary Random Environments} \label{Background}

Fix $d\geq 1$.
We work on the $d$-dimensional Euclidean lattice $(\Z^d,E_d),$ where $E_d=\{\{x,y\}\in\Z^d\times\Z^d:\norm{x-y}_1=1\}.$ We write $x\sim y$ if $\{x,y\}\in E_d,$ and $B(x,r)$ for the $l^1$ ball centred at $x$ with radius $r.$ For each $e=\{x,y\} \in E_d$ and $z \in \Z^d$, we write $e-z$ for the edge $\{x-z,y-z\}$.
We define an \textbf{evironment} to be a non-negative element of the space $L_\mathrm{loc}^1(E_d \times \R)$ of locally integrable, measurable functions $E_d \times \R \to \R$ modulo a.e.\ equivalence, where we recall that $f:E_d \times \R \to \R$ is said to be \textbf{locally integrable} if $\int_a^b |f_t(e)| \dif t<\infty$ for every $a<b$ and every edge $e\in E_d$.
(Here and elsewhere we follow the usual convention of writing the time variable as a subscript.)
 We recall that $L_\mathrm{loc}^1(E_d \times \R)$ can be endowed with a unique topology, called the \emph{local $L^1$ topology}, with the property that $f^n$ converges to $f$ if and only if $\int_a^b |f^n_t(e)-f_t(e)| \dif t\to 0$ as $n\to\infty$ for every $a<b$ and $e\in E_d$.
 We write $\Omega = \{\eta \in L_\mathrm{loc}^1(E_d \times \R) : \eta_t(e) \geq 0 $ for every $e\in E_d$ and a.e.\ $t \in \R\}$ for the space of environments, which we equip with the associated subspace topology and Borel $\sigma$-algebra. For each environment $\eta\in \Omega$ and $x\in \Z^d$ we write $\eta_t(x)=\sum_{y\sim x} \eta_t(\{x,y\})$.

We refer to a random variable taking values in $\Omega$ as a \textbf{random environment}. For each $x\in \Z^d$ and $t\in \R$ we write $\tau_{x,t}:\Omega\to\Omega$ for the space-time shift defined by $\tau_{x,t}\eta_s(e)= \eta_{s-t}(e-x)$ and say that a random environment $\eta$ is \textbf{stationary} if 
$\tau_{x,t} (\eta)$ has the same distribution as $\eta$ for every $x\in \Z^d$ and $t\in \R$. Similarly, we define the time-reversal map $R:\Omega \to \Omega$ by $R(\eta)_{t}(e)=\eta_{-t}(e)$ and say that a random environment $\eta$ is \textbf{reversible} if $R(\eta)$ has the same distribution as $\eta$.
For each $a<b,$ let $\mathcal{F}_{[a,b]}$ be the $\sigma$-algebra generated by the restriction of $\eta$ to $[a,b]$. We say that $\eta$ is a \textbf{Markovian random environment} if $\mathcal{F}_{[a_1,a_2]}$ and $\mathcal{F}_{[c_1,c_2]}$ are conditionally independent given $\mathcal{F}_{[b_1,b_2]}$ whenever $a_2<b_2$ and $c_1>b_1$ (that is, if the past and the future are conditionally independent given the present). For each $t\in \R$, we define the \textbf{instantaneous sigma-algebra} $\mathcal{F}_{t} = \bigcap \{\mathcal{F}_{[a,b]} : a< t < b \}$, and say that $\eta$ is \textbf{strongly reversible} if the conditional distributions of $\eta$ and $R(\eta)$ given $\mathcal{F}_0$ are the same almost surely. For example, if $\theta$ is a uniform random element of $[0,2\pi],$ then the environment $\eta$ defined by $\eta_t(e)=(\sin(t+\theta))_{t\in \R}$ for every $e\in E_d$ and $t\in \R$ is a stationary reversible Markovian environment that is not strongly reversible.

 Let $\Z^d_\infty = \Z^d \cup \{\infty\}$ be the one-point compactification of $\Z^d$ and let $D(\R,\Z^d_\infty)$ be the space of $\Z^d_\infty$-valued càdlàg functions on $\R$, which we equip with the Skorohod topology and associated Borel $\sigma$-algebra. 
The point at infinity is included to deal with the possibility of an explosion.
For each starting space-time location $(u,s)\in\Z^d \times \R$ and environment $\eta \in \Omega,$ there exists a unique probability measure $\mathbb{P}_{u,s}^\eta$ on $D(\R,\Z^d_\infty)$ under which the coordinate process $(X_t)_{t \in \R}$ is a an inhomogeneous continuous time Markov Chain on $\Z^d$ starting at $u$ at time $s$ and with self-adjoint time-dependent generator $(\mathcal{L}_{t}^{\eta})_{t\in \R}$ defined by
\[
\mathcal{L}_{t}^{\eta} f(x)=\sum_{y \sim x} \eta_t(\{x,y\})(f(y)-f(x)).
\]
We denote the transition probabilities of this Markov chain by $P^\eta_{t_1,t_2}(u,v)=\mathbb{P}^\eta_{u,t_1}(X_{t_2}=v)$ for each $t_1, t_2$ and $u,v\in \Z^d$. We say that an environment $\eta$ is \textbf{non-explosive} if $\mathbb{P}_{u,s}^\eta$ is supported on paths that make at most finitely many jumps in any bounded interval of time for every $u\in V$ and $s\in \R$.

\medskip

\textbf{A Poissonian reformulation.} As usual, one can equivalently define the random walk in the environment $\eta$ using Poisson processes rather than generators. We first briefly recall how point processes in $E_d \times \R$ can be used to define walks.
Let $\mathcal{D}$ be the set of subsets $U\subset \R \times E_d$ that are discrete (i.e. consist only of isolated points), and for which $U \cap (E_d \times \{t\})$ contains at most one point for each $t\in \R$. For each $U \in \mathcal{D}$, let $J=J(U)$ be the set of space-time points $(u,t) \in \Z^d \times \R$ such that $(\{u,v\},t)\in U$ for some neighbour $v$ of $u$.
%
Given $U\in \mathcal{D}$ and a space-time coordinate $(u,t)\notin J(U)$, we define the induced cádlág path $F_{u,t}(U)=(F_{u,t}(U)_s)_{s\in \R} \in D(\R,\Z^d)$ which starts with $F_{u,t}(U)=u$ and follows the points of $U$ forwards and backwards in time, traversing an edge $e=\{x,y\}$ at time $s\geq t$ if $\lim_{\varepsilon \downarrow 0} F_{u,t}(U)_{s-\varepsilon} \in\{x,y\}$ and $(e,s)\in U$ and, similarly,
traversing an edge $e=\{x,y\}$ at time $s\leq t$ if $\lim_{\varepsilon \downarrow 0} F_{u,t}(U)_{s+\varepsilon} \in\{x,y\}$ and $(e,s)\in U$.
 We define $T_\infty+$ and $T_\infty^-$ to be the forward and backward explosion times of $F_{u,t}(U)$, and set $F_{u,t}(U)_s=\infty$ for all $s \geq T_\infty^+$ and $s \leq T_\infty^-$. 
 
\medskip

\textbf{Translation and reflection equivariance.}
 An important property of this construction is that for any $U \in \mathcal{D}$ and any two space-time points $(u,s),(v,t) \in (\Z^d \times \R) \setminus J(U)$ we have that
 \begin{equation}
 \label{eq:biject}
F_{u,s}(U)_t = v \iff F_{v,t}(U)_s=u \iff F_{u,s}(U)=F_{v,t}(U).
 \end{equation}
Indeed, if we start a particle at $(u,s)$ then follow the points of $U$ forwards in time until we hit $v$ at time $t \geq s$, then if we instead start at $v$ at time $t$ and follow the points of $U$ backwards in time until time $s,$ we will end up at $u$.
A further important property of the map $F:\mathcal{D}\times \Z^d \times \R$ is that it is equivariant with respect to space-time shifts and time-reversals. That is, if we define the space-time shifts 
\begin{align*}
\tau_{x,t} : \mathcal{D} &\longrightarrow \mathcal{D} & \tau_{x,t} : D(\R,\Z^d_\infty) &\longrightarrow D(\R,\Z^d_\infty)\\
U &\longmapsto \bigl\{(e-x,s-t) : (e,s) \in U\bigr\}
& (\zeta_s)_{s\in \R} &\longmapsto (\zeta_{s-t}-x)_{s\in \R}
\end{align*}
for each $x \in \Z^d$ and $t\in \R$ and the time-reversal maps
\begin{align*}
R : \mathcal{D} &\longrightarrow \mathcal{D} & R : D(\R,\Z^d_\infty) &\longrightarrow D(\R,\Z^d_\infty)\\
U &\longmapsto \bigl\{(e,-s) : (e,s) \in U\bigr\}
& (\zeta_s)_{s\in \R} &\longmapsto \Bigl(\lim_{\varepsilon \downarrow 0} \zeta_{-s+\varepsilon}\Bigr)_{s\in \R}
\end{align*}
then we have that
\[
\tau_{x,t}(F_{u,s}(U)) = F_{u-x,s-t}(\tau_{x,t} (U)) \qquad \text{ and }  \qquad 
R (F_{u,s}(U)) = F_{u,-s}(R (U))
\]
for every $(x,t) \in \Z^d \times \R$, $U \in \mathcal{D}$, and $(u,s)\in (\Z^d \times \R) \setminus J(U)$.

 Given an environment $\eta$, we may take $U$ to be the inhomogeneous Poisson process on $E_d \times \R$ with intensity $\eta$, which belongs to $\mathcal{D}$ almost surely since $\eta$ is locally integrable. It is a standard and easily verified fact that the resulting process $F_{u,t}(U)$ then has law $\mathbb{P}^\eta_{u,t}$ for each $u\in \Z^d$ and $t\in \R$. 
Fixing $\eta$ and taking expectations over $U$ in \eqref{eq:biject} therefore yield the detailed-balance equations
\begin{equation} \label{eq:rev}
P_{s,t}^\eta(u,v)=P_{t,s}^\eta(v,u),
\end{equation}
which also follow directly by self-adjointness of the generators. Moreover, if 
$U$ is a Poisson process with intensity $\eta$, then $R(U)$ is a Poisson process with intensity $R(\eta)$, and it follows that if $X=(X_t)_{t\in \R}$ has law $\mathbb{P}_{u,s}^\eta$, then $R(X)$ has law $\mathbb{P}_{u,-s}^{R(\eta)}$. It follows in particular that if $\eta$ is a stationary reversible random environment and $X=(X_t)_{t\in \R}$ is the associated random walk started at $(u,s)$, then $X$ and $R(X)$ have the same marginal distribution (the \emph{conditional} distributions of these processes given $\eta$ need not be the same).

\subsection{Moment conditions}
\label{sec:Moments}

Let $d\geq 1$ and let $\eta \in \Omega$ be a stationary random environment on $\Z^d$. Recall that we write $\eta_t(x):=\sum_{y\sim x} \eta_t(\{x,y\})$ for the total conductance of all edges incident to $x$ at time $t$. For each  $p\geq 1$ we define the \textbf{infinitesimal $p$-norm} $\|\eta\|_p$ of $\eta$  to be
\begin{equation*} \label{eq:mom1}
\|\eta\|_p:=\sup_{[a,b]\subset\R} \frac{1}{b-a} 
\mathbb{E}\Bigg[\Bigl(\int_a^b \eta_s(0) \diff s\Bigr)^p\Bigg]^{1/p} = \limsup_{\varepsilon \downarrow 0} \frac{1}{\epsilon} 
\mathbb{E}\Bigg[\Bigl(\int_{[0,\epsilon]} \eta_s(0) \dif s\Bigr)^p\Bigg]^{1/p},
\end{equation*}
where the equivalence of these two quantities follows by stationarity and Minkowski's inequality. Note that $\|\eta\|_p$ is increasing in $p\geq 1$ and that if $\eta$ is, say, bounded and a.s.\ c\'adl\'ag, so that $\eta_t(x)$ is well-defined pointwise, then $\|\eta\|_p=\|\eta_t(x)\|_p$ for every $x\in \Z^d$ and $t\in \R$. 

The next proposition shows that first and second moment bounds on the total conductance at a \emph{fixed} vertex imply first and second moment bounds on the number of times the walk jumps. We will deduce in particular that $\|\eta\|_1<\infty$ is a sufficient condition for non-explosivity, generalising \cite[Lemma 4.1]{andres2018}. For each two integers $p \geq 1$ and $1 \leq \ell \leq p$, we write ${ p \brace \ell}$ for the Stirling numbers of the second kind, which are defined to be the unique non-negative integers such that $x^p = \sum_{\ell=1}^p  {p \brace \ell} \ell! \binom{x}{\ell}$ for every $x\in \R$. (Equivalently, ${p \brace \ell}$ is the number of ways to partition a set of size $p$ into $\ell$ non-empty subsets.)

\begin{proposition} \label{prop:ticks} Let $d\geq 1$, let $\eta$ be a stationary random environment on $\Z^d$, let $(u,s)\in \Z^d \times \R$ be a space-time location, and let $X=(X_t)_{t\in \R}$ be the associated random walk started at the origin at time zero.
For each $0\leq a < b$ let $N[a,b]$ denote the cardinality of the set of jump times
	$\{t\in[a,b]:X_{t^-} \neq X_t \}$.
Then
	\[
\mathbb{E}\left[N[a,b]^p\right] \leq \sum_{\ell=1}^p  { p \brace \ell } \ell! |a-b|^\ell \|\eta\|_{\ell}^\ell
	\]
	for every integer $p \geq 1$.
In particular, if $\|\eta\|_1<\infty,$ then $\eta$ is non-explosive almost surely.
\end{proposition}

The most important consequence of this theorem is the statement that if $\|\eta\|_p<\infty$ for some integer $p\geq 1,$ then $\mathbb{E}\left[N[a,b]^p\right] <\infty$ for every $a<b$. We will only use the cases $p=1,2$ of this proposition, but prove the general case for possible future applications since it is not much more work.

The proof of Proposition~\ref{prop:ticks} will rely on the construction of the \emph{censored random walk} in finite volume, which we now introduce. Let $\eta$ be a stationary random environment on $\Z^d$, 
let $U$ be a Poisson process with intensity $\eta$, and  
let $X=F_{0,0}(U)$ be the associated random walk in $\eta$ started at $(0,0)$.
Consider the sequence of $l_1$ boxes $B_k = B(0,k)\cap\Z^d$ for $k\geq 1$, and let $E_{d,k}$ be the set of edges of $\Z^d$ with both endpoints in $B_k$.
For each $k\geq 1,$ let $S_k$ be a uniform random element of $B_k$ independent of $\eta$ and $U$, and let $X^k=F_{S_k,0}(U)$ be a random walk in $\eta$ started at $(S_k,0)$. Stationarity of $\eta$ implies that $X^k-S_k=(X^k_t-S_k)_{t\in \R}$ and $X$ have the same distribution for every $k\geq 1$.
 
For each $k\geq 1$, let $U^k = U \cap (E_{d,k} \times \R)$, and define the \textbf{censored random walk} $Z^k=(Z^k_t)_{t\in \R}=F_{S_k,0}(U^k)$. 
In other words, the censored random walk $Z^k$ is coupled with the random walk $X^k$ by setting $Z^k_0=X^k_0,$ and then letting $Z^k$  follow the same Poisson point process $U$ as $X^k,$ forwards and backwards in time, but ignoring the edges which lead out of $B_k$. Thus, $Z^k$ is guaranteed to equal to $X^k$ up until the first time $X^k$ leaves the ball $B_k.$
Observe that censored random walks cannot explode since the rate of transition of the walk at any time is bounded above by the total conductance of all the edges contained within the box, which is finite by assumption.

Note that if $\eta$ is a stationary Markovian random environment and $k\geq 1,$ then both $(\eta,X)$ and $(\eta,Z^k)$ are Markov processes in the sense that the future and the past are conditionally independent given the present; see Section~\ref{subsec:Markovtype} for details. However, the censored random walk has the advantage that the associated Markov process admits a stationary \emph{probability} measure. Indeed, we will argue more generally that if $\eta$ is a stationary random environment then $(\eta,Z^k)$ is time-stationary in the sense $\tau_{0,t}(\eta,Z^k):=(\tau_{0,t} (\eta), \tau_{0,t} (Z^k))$ has the same distribution as $(\eta,Z^k)$ for every $k\geq 1$ and $t\in \R$.

\begin{lemma} \label{lem:censrevstat}
	Let $d\geq 1$ and let $\eta$ be a stationary random environment. Then the processes $(\eta_t,Z^k_t)_{t\in\R}$ are stationary for each $k\geq 1$. 
\end{lemma}
	
	\begin{proof}
		Fix $k\geq 1$.
		Let $U$ be a Poisson process with intensity $\eta$ and let $U^k$ be defined as above. We have by \eqref{eq:biject} that if $(u,s),(v,t)\notin J(U),$ then
 \begin{equation}
 \label{eq:biject_recall}
F_{u,s}(U^k)_t = v \iff F_{v,t}(U^k)_s=u \iff F_{u,s}(U^k)=F_{v,t}(U^k).
 \end{equation}
		 One implication of this is that for any $s,t \in \R$, the function $\sigma_{s,t}:B_k\rightarrow B_k$ given by $\sigma_{s,t}(u)=[F_{s,u}(U^k)]_t$ is  almost surely a bijection with the property that
		\[
			F_{s,u}(U^k)=F_{t,\sigma_{s,t}(u)}(U^k)
		\]
		for every $u\in B_k$.
		Letting $S_k$ be a uniform random element of $B_k$ independent of $\eta$ and $U$, we deduce that $S_k$ and $\sigma_{s,t}(S_k)$ have the same conditional distribution given $\eta$ and $U$ and hence that
		\[
		\tau_{0,t} \Bigl( \eta, F_{0,S_k}\bigl(U^k\bigr) \Bigr) =
\tau_{0,t} \Bigl( \eta, F_{t,\sigma_{0,t}(S_k)}\bigl(U^k\bigr) \Bigr)
\sim 
\tau_{0,t} \Bigl( \eta, F_{t,S_k}\bigl(U^k\bigr) \Bigr)
		 \sim \left(\eta,F_{0,S_k}\bigl(U^k\bigr)\right)\]
		 for every $t\in \R$, where we used stationarity of $\eta$ and shift-equivariance of $F$ in the final equality in distribution. This completes the proof of stationarity.
 \qedhere
	\end{proof}

We will deduce Proposition~\ref{prop:ticks} from the following analogous statement for the censored random walk.

\begin{lemma} \label{lem:censoredticks} Let $d\geq 1$, let $\eta$ be a stationary random environment on $\Z^d$, let $k \geq 1$, and let $Z^k$ be the censored random walk in $\eta$.
For each $0\leq a < b$ let $N_k[a,b]$ denote the cardinality of the set of jump times
	$\{t\in[a,b]:Z^k_{t^-} \neq Z^k_t \}$. Then
	\[
\mathbb{E}\left[N_k[a,b]^p\right] \leq \sum_{\ell=1}^p  { p \brace \ell } \ell! |a-b|^\ell \|\eta\|_{\ell}^\ell
	\]
for every integer $p \geq 1$.
\end{lemma}

	\begin{proof}
		 By stationarity, we can without loss of generality assume that $a=0$. We fix $b\geq 0$ and $k\geq1,$ and write $N=N_k,$ and $Z=Z^k$. 
		For each $n\in\N$ and $i\in\Z$, define \[A_{i,n}=\mathbbm{1}\left(N\left[\frac{(i-1)b}{n},\frac{ib}{n}\right]> 0\right) \qquad \text{ and } \qquad \Sigma_n = \sum_{i=1}^n A_{i,n}.
		\]
		Since $N=\lim_{n\rightarrow\infty}\Sigma_n$ almost surely and $(\Sigma_{2^n})$ is a monotone increasing sequence, the monotone convergence theorem implies that
		$\E{N^p} = \lim_{n\to\infty} \E{\Sigma_{2^n}^p}$ for every $p\geq 0$. 
		Since $\E{\Sigma_{2^n}^p}=\sum_{\ell=1}^p {p \brace \ell} \ell! \mathbb{E}\binom{\Sigma_{2^n}}{\ell}$
		it therefore suffices to prove that
		\[
\mathbb{E} \binom{\Sigma_n}{\ell} = \sum_{1\leq i_1 < \cdots < i_q \leq n} \mathbb{E} \Bigl[\prod_{j=1}^\ell A_{i_j,n}\Bigr] \leq b^\ell \|\eta\|_\ell^\ell
		\]
		for every $\ell \geq 1$.
		Writing $\mathbf{E}^\eta$ for expectations conditional on the environment $\eta$ and the uniform starting point $S_k =Z^k_0 \in B_k$,
		we will prove by induction on $\ell$ that the stronger inequality 
		\begin{equation}
		\mathbb{E}\left[\mathbf{E}^\eta\Bigl[\prod_{j=1}^\ell A_{i_j,n}\Bigr]^q\right] \leq \left(\frac{b\|\eta\|_{q\ell}}{n}\right)^{q\ell}
		\label{eq:induction_hypothesis}
		\end{equation}
		holds for every $n\geq 1$, $\ell \geq 0$, $q\geq 1$, and every increasing sequence $i_1<i_2<\ldots<i_\ell$ in $\R$, where we take the empty product to be $1$. (Note that we do not assume that $q$ is an integer.) 

		The $\ell=0$ case holds vacuously. Assume that the claim holds for some $\ell \geq 0$ and let $i_0<\ldots<i_\ell$ be an increasing sequence of times. Then we have by stationarity (Lemma~\ref{lem:censrevstat}) and the fact that $(Z^k_t)_{t \leq 0}$ and $(Z^k_t)_{t \geq 0}$ are conditionally independent given $\eta$ and $S_k$ that
		\begin{align*}
		\mathbb{E}\left[\mathbf{E}^\eta\Bigl[\prod_{j=0}^\ell A_{i_j,n}\Bigr]^q\right]
		=
		\mathbb{E}\left[\mathbf{E}^\eta\Bigl[\prod_{j=0}^\ell A_{i_j-i_0,n}\Bigr]^q\right]
		\leq
		\mathbb{E}\left[\mathbf{E}^\eta[A_{0,n} ]^q \cdot \mathbf{E}^\eta\Bigl[\prod_{j=1}^\ell A_{i_j,n}\Bigr]^q\right].
		\end{align*}
		Applying H\"older's inequality and the induction hypothesis yields that
		\begin{multline}
		\label{eq:Holder_induction}
		\mathbb{E}\left[\mathbf{E}^\eta\Bigl[\prod_{j=0}^\ell A_{i_j,n}\Bigr]^q\right]
		\leq
		\mathbb{E}\Bigl[\mathbf{E}^\eta[A_{0,n}]^{q(\ell+1)} \Bigr]^{1/(\ell+1)} \mathbb{E}\left[ \mathbf{E}^\eta\Bigl[\prod_{j=1}^\ell A_{i_j,n}\Bigr]^{q(\ell+1)/\ell}\right]^{\ell/(\ell+1)} 
		\\
		\leq \left(\frac{b\|\eta\|_{q(\ell+1)}}{n}\right)^{q\ell}\mathbb{E}\Bigl[\mathbf{E}^\eta[A_{0,n}]^{q(\ell+1)} \Bigr]^{1/(\ell+1)}.
		\end{multline}
		Conditioned on $\eta$ and $Z_0=S_k$, the indicator random variable $A_{0,n}$ is equal to $1$ if and only if at least one of the Poisson clocks attached to an edge incident to $Z_0$ rings in the interval $[-b/n,0]$, so that
		\[
		\mathbf{E}^\eta\left[A_{0,n}\right] = 1-\exp{\left[-\int_{-b/n}^{0}\eta_t(Z_0)\dif t\right]}
		\leq \int_{-b/n}^{0}\eta_t(Z_{0})\dif t
		\]
		and hence by stationarity of $\eta$ that
		\[
		\mathbb{E}\Bigl[\mathbf{E}^\eta[A_{0,n}]^{q(\ell+1)} \Bigr] \leq 
		 \mathbb{E}\left[\left(\int_{-b/n}^{0}\eta_t(Z_{0})\dif t\right)^{q(\ell+1)}\right] \leq \left(\frac{b \|\eta\|_{q(\ell+1)}}{n}\right)^{q(\ell+1)}.
		\]
		Substituting this estimate into \eqref{eq:Holder_induction} completes the induction step and hence the proof of \eqref{eq:induction_hypothesis}.
\qedhere
	\end{proof}

\begin{proof}[Proof of Proposition~\ref{prop:ticks}]
	Fix $b>0.$ Lemma \ref{lem:censoredticks} implies that the first moment of $\max_{0\leq t\leq b}d(Z^k_t,S_k)\leq N_k[0,b]$ is bounded above uniformly in $k$. We also note that for any fixed distance $l>0$ the probability that the distance between $S_k$ and the boundary of $B_k$ is less than $l$ decreases to zero as $k$ tends to infinity. Combining these two observations, the probability that $Z^k-S_k$ hits the boundary of $B_k-S_k$ before time $b$ tends to zero as $k\rightarrow\infty.$ Since $Z^k$ and $X^k$ are equal up to the first time the boundary is hit, and, by stationarity, the law of  $(X^k_t-S_k)_{0\leq t\leq b}$ is equal to the law of $(X_t)_{0\leq t\leq b},$ it follows that $(Z^k_t-S_k)_{0\leq t\leq b}$ converges in distribution to $(X_t)_{0\leq t\leq b}$ as $k\to\infty$.
	It follows that the law of $(N_k[a,b])_{0\leq a\leq b}$ converges weakly to the law of $(N[a,b])_{0\leq a\leq b},$ and hence by Fatou's lemma that
	\[\E{N[a,b]^p}\leq\liminf_{k\rightarrow\infty} \E{N_k[a,b]^p}\leq\sum_{\ell=1}^p  { p \brace \ell } \ell! |a-b|^\ell \|\eta\|_{\ell}^\ell,\] where the second inequality follows by Lemma \ref{lem:censoredticks}.
\end{proof}

\subsection{Diffusive upper bounds via Markov-type inequalities}
\label{subsec:Markovtype}

In this section we use \emph{Markov-type inequalities} to establish diffusive upper bounds on the displacement of random walks in stationary reversible Markovian environments, generalising an argument of Peres, Stauffer, Steif \cite[Theorem 1.9]{Peres2015} from the setting of dynamical percolation. To do this, we will need a version of the Markov-type inequality that applies to Markov processes defined on uncountable state spaces and that need not be well-defined pointwise. The proof of this inequality is in fact very similar to the usual discrete-time proof of Naor, Peres, and Sheffield \cite{MR2239346} as presented in \cite[Lemma 13.15]{MR3616205}. Markov-type inequalities were first studied by Keith Ball in his work on the Lipschitz extension problem \cite{MR1159828}, and have recently found many important applications in probability theory including e.g.\ \cite{gwynne2020anomalous,ganguly2020chemical,lee2020relations,Peres2015,MR4073931}.

 We now introduce the relevant definitions.
Let $\mathbb{X}$ be a Polish space, and let $\mathcal{Z}=\mathcal{Z}(\R,\mathbb{X})$ be the set of Borel-measurable functions from $\R$ to $\mathbb{X}$ modulo almost-everywhere equivalence.
For each $s\in \R$ we define the time-shift $\tau_s : \mathcal{Z} \to \mathcal{Z}$ by $\tau_s \zeta(t)=\zeta(t-s)$ for every $\zeta\in \mathcal{Z}$ and $t\in \R$, and define the reversal $R:\mathcal{Z} \to \mathcal{Z}$ by $R(\zeta)(t)=\zeta(-t)$ for every $\zeta\in \mathcal{Z}$ and $t\in \R$.
 Let $Z$ be a random variable taking values in $\mathcal{Z}$, and for each $a<b$ let $\mathcal{F}_{[a,b]}$ be the $\sigma$-algebra generated by the restriction of $Z$ to $[a,b]$. We say that $Z$ is a \textbf{Markov process} if $\mathcal{F}_{[a_1,a_2]}$ and $\mathcal{F}_{[c_1,c_2]}$ are conditionally independent given $\mathcal{F}_{[b_1,b_2]}$ whenever $a_2<b_2$ and $c_1>b_1$ (that is, if the past and the future are conditionally independent given the present). We say that $Z$ is \textbf{stationary} if $\tau_s \mathcal{Z}$ has the same distribution as $Z$ for every $s\in \R$, and that $Z$ is \textbf{reversible} if $R(Z)$ and $Z$ have the same distribution. For each $t\in \R$, we define the \textbf{instantaneous sigma-algebra} $\mathcal{F}_{t} = \bigcap \{\mathcal{F}_{[a,b]} : a< t < b \}$, and say that $Z$ is \textbf{strongly reversible} if the conditional distributions of $Z$ and $R(Z)$ given $\mathcal{F}_0$ are the same almost surely.

\begin{proposition}[Generalised maximal Markov-type inequality] \label{prop:MTIN}
	Let $\mathbb{X}$ be a Polish space, and let $Z \in \mathcal{Z}(\R,\mathbb{X})$ be a stationary, strongly reversible Markov process. Let $d \geq 1$ and let $f:\mathcal{Z}\to \R^d$ be measurable with respect to the instantaneous sigma-algebra $\mathcal{F}_{0}$ and reversible in the sense that $f(Z)=f(R(Z))$ almost surely.	
	Then we have that
	\begin{equation} \label{eq:discreteGen}
	\mathbb{E}\left[ \max_{0\leq m \leq n} \|f(\tau_{2mt} Z)-f(Z)\|_2 ^2\right] \leq 25n \mathbb{E}\Bigl[  \|f(\tau_t Z)-f(Z)\|_2 ^2\Bigr].
	\end{equation}
	for every $n \geq 1$ and $t>0$ and hence that
	\begin{equation} \label{eq:ctsGen}
	\E{\esssup_{0\leq s\leq t}\|f(\tau_{s} Z)-f(Z)\|_2 ^2}\leq \frac{25 t}{2}  \limsup_{\eps \downarrow 0}\frac{1}{\eps}\mathbb{E}\left[\|f(\tau_{\eps}Z)-f(Z)\|_2^2\right]\end{equation}
	for every $t>0$.
\end{proposition}

\begin{remark}
If $\theta$ is a uniform random element of $[0,2\pi]$ then $(X_t)_{t\in \R}=(\sin(t+\theta))_{t\in \R}$ is a stationary reversible Markov process $X:\R \to \R$ that is not strongly reversible and does not satisfy the conclusions of the Markov-type inequality. Indeed, if we consider the identity function $f:\R\to \R$ then
\[
\mathbb{E}\left[ \|f(X_t)-f(X_0)\|_2^2 \right] = \int_0^{2\pi} \left[\sin(t+\theta)-\sin(\theta)\right]^2 \dif \theta = 2\pi \left(1-\cos(t)\right)= \Theta(t^2) \quad \text{ as $t\downarrow 0$},
\]
so that $\mathbb{E}\left[ \|f(X_{nt})-f(X_0)\|_2^2 \right] \gg n \mathbb{E}\left[ \|f(X_{t})-f(X_0)\|_2^2 \right]$ when $t$ is small and $n$ is large. Further processes with similar properties include e.g.\ \emph{piecewise deterministic Markov processes} and the \emph{integrated Ornstein-Uhlenbeck process mod $1$}.
\end{remark}

\begin{proof}[Proof of Proposition~\ref{prop:MTIN}]
	Without loss of generality we may take $d=1$, the higher-dimensional cases following by summing the inequalities \eqref{eq:discreteGen} and \eqref{eq:ctsGen} over the coordinates of  $f$.
	We may also assume that $f$ is bounded, truncating $f$ to $[-r,r]$ and using monotone convergence to take the limit as $r\to\infty$ otherwise.  Note that if $\theta$ is a uniform random number in $[1/2,1]$ and $N=N(\theta,n)=\lceil nt/2\theta \rceil$ for each $n\geq 1$ then $\max_{0\leq m\leq N}\|f(\tau_{2m \theta t/n} Z)-f(Z)\|_2^2$ converges in probability to $\esssup_{0\leq s\leq t}\|f(\tau_{s} Z)-f(Z)\|_2 ^2$ as $n\to\infty$ (this follows by e.g.\ the Lebesgue differentiation theorem), so that \eqref{eq:ctsGen} follows from \eqref{eq:discreteGen} and Fatou's lemma.
	
	The main idea, taken from \cite{MR2239346}, is to write the maximum we are interested in terms of two martingales, one going forwards in time and the other backwards in time, and then use Doob's $L^2$ maximal inequality. 
	For each $t \in \R$, let $\mathcal{G}_t^\rightarrow =\bigcap_{s>t} \mathcal{F}_{(-\infty,s]}$ and let $\mathcal{G}_t^\leftarrow =\bigcap_{s<t} \mathcal{F}_{[s,\infty)}$, so that $\mathcal{F}_t \subseteq \mathcal{G}_t^\rightarrow \cap \mathcal{G}_t^\leftarrow$ for each $r\in \R$. Since $Z$ is a Markov process, $\mathcal{F}_s$ and $\mathcal{G}_t^\rightarrow$ are conditionally independent given $\mathcal{F}_t$ when $s>t$, while $\mathcal{F}_s$ and $\mathcal{G}_t^\leftarrow$ are conditionally independent given $\mathcal{F}_t$ when $s<t$.
	Fix $t>0$ and $n\in\N,$ and for  each $1\leq m\leq 2n$ let 
	\[D_m^\rightarrow= f(\tau_{mt}Z)-\E{f(\tau_{mt}Z)\mid \mathcal{G}_{(m-1)t}^\rightarrow} = f(\tau_{mt}Z)-\E{f(\tau_{mt}Z)\mid\mathcal{F}_{(m-1)t}},\]
	where the almost-sure equivalence of these two quantities follows from the assumption that $Z$ is a Markov-process and that $f$ is $\mathcal{F}_0$-measurable. 
	In particular, the process $( D_i^\rightarrow)_{m=1}^{2n}$ is a martingale difference sequence with respect to the filtration $(\mathcal{G}_{mt}^\rightarrow)_{m=0}^{n}$.
	Similarly, for each $1 \leq m \leq 2n$ we define
	\[D_m^\leftarrow
	=f(\tau_{(2n-m)t}Z)-\E{f(\tau_{(2n-m)t}Z)\mid\mathcal{G}^\leftarrow_{(2n-m+1)t}}
	=f(\tau_{(2n-m)t}Z)-\E{f(\tau_{(2n-m)t}Z)\mid\mathcal{F}_{(2n-m+1)t}}.\]
	As before, the almost-sure equivalence of these quantities follows
	from the assumption that $Z$ is a Markov-process and that $f$ is $\mathcal{F}_0$-measurable.
	In particular, the process $(D_m^\leftarrow)_{m=1}^{2n}$ is a martingale difference sequence with respect to the filtration $(\mathcal{G}_{(2n-m)t}^\leftarrow)_{m=0}^{n}$.
%
	Moreover, for each $2 \leq m \leq 2n$ we have that \begin{multline} \label{eq:diff} D_{m}^\rightarrow-D^\leftarrow_{2n-m+2}=f(\tau_{mt}Z)-f(\tau_{(m-2)t}Z)-\E{f(\tau_{mt}Z)-f(\tau_{(m-2)t}Z)\mid\mathcal{F}_{(m-1)t}}\\=f(\tau_{mt}Z)-f(\tau_{(m-2)t}Z)\end{multline}
	almost surely, where we used stationarity and strong reversibility to deduce  that
$f(\tau_{mt}Z)$ and $f(\tau_{(m-2)t}Z)$ have the same conditional distribution given $\mathcal{F}_{(m-1)t}$ almost surely and hence that 
 the central conditional expectation is almost surely zero. We obtain by algebra that
\begin{align*}
f(\tau_{2kt} Z) - f(Z) = \sum_{m=1}^k D_{2m}^\rightarrow - \sum_{m=1}^k D_{n-2m+2}^\leftarrow
\end{align*}
for every $1 \leq k \leq n$. It follows that
\begin{multline*}
\max_{0\leq k \leq n} |f(\tau_{2kt}Z)-f(Z)| \leq \max_{0\leq k \leq n} \left|\sum_{m=1}^k D_{2m}^\rightarrow \right|+\max_{0\leq k \leq n}\left|\sum_{m=1}^k D_{2n-2m+2}^\leftarrow\right|
\\\leq 
\max_{0\leq k \leq n} \left|\sum_{m=1}^k D_{2m}^\rightarrow \right|+\max_{0\leq k \leq n}\left|\sum_{m=1}^k D_{2m}^\leftarrow\right| + \left|\sum_{m=1}^n D_{2m}^\leftarrow\right|
\end{multline*}
and hence by Cauchy-Schwarz that
\[
\max_{0\leq k \leq n} |f(\tau_{2kt}Z)-f(Z)|^2 
\leq 
\frac{5}{2}\max_{0\leq k \leq n} \left|\sum_{m=1}^k D_{2m}^\rightarrow \right|^2+\frac{5}{2}\max_{0\leq k \leq n}\left|\sum_{m=1}^k D_{2m}^\leftarrow\right|^2 + 5 \left|\sum_{m=1}^n D_{2m}^\leftarrow\right|^2.
\]
Applying Doob's $L^2$ maximal inequality and the orthogonality of martingale differences, we obtain that
\begin{align*}
\E{\max_{0\leq k \leq n} |f(\tau_{2kt}Z)-f(Z)|^2 } &\leq 10 \sum_{m=1}^n \E{\left(D_{2m}^\rightarrow\right)^2} + 15 \sum_{m=1}^n \E{\left(D_{2m}^\leftarrow\right)^2}.
\end{align*}
Using stationarity and reversibility once more, we obtain that 
\begin{align*}
\E{\max_{0\leq k \leq n} |f(\tau_{2kt}Z)-f(Z)|^2 }
&\leq 25 n \E{\left(f(\tau_{t}Z)-\E{f(\tau_t Z) \mid \mathcal{F}_0}\right)^2} 
\\&=25n \E{\left(f(\tau_{t}Z)-f(Z) - \E{f(\tau_t Z) -f(Z) \mid \mathcal{F}_0}\right)^2} 
\\&= 25n \E{\operatorname{Var}(f(\tau_t Z)-f(Z) \mid \mathcal{F}_0)} \leq 25n \operatorname{Var}(f(\tau_t Z)-f(Z)),
\end{align*}
which implies the claim. \qedhere
\end{proof}

Proposition~\ref{prop:MTIN} has the following corollary for random walks in reversible random environments.

\begin{corollary}\label{prop:mTDist}
Let $d\geq 1$, let $\eta$ be a stationary, strongly reversible Markovian random environment on $\Z^d$ and let $X=(X_t)_{t\in \R}$ be the associated random walk started at the origin at time zero. If $\norm{\eta}_2<\infty$, then 
	\[\E{\max_{-t\leq s\leq t} \|X_s-X_0\|^2_2}\leq 25t\norm{\eta}_1\] 
	for every $t\geq 0.$
\end{corollary}

	\begin{proof}
		Let $k\geq 1$ and let $(Z^k_t)_{t\in \R}$ be the censored random walk started at a uniform random element $S_k$ of $B_k$ as in Section~\ref{sec:Moments}. By Lemma~\ref{lem:censrevstat}, $(\eta_t,Z^k_t)$ is a stationary Markov process. Moreover, if we consider this process to take values in the space of measurable functions $\mathcal{Z}=\mathcal{Z}(\R,\R^{E_d}\times \Z^d)$ then it is strongly reversible: this follows by time-reversal equivariance of $F$ and the fact that, given $\eta$, the reversed Poisson process $R(U)$ has the same conditional distribution as a Poisson process with intensity $R(\eta)$.
		 Thus, we may apply Proposition~\ref{prop:MTIN} to the function $f: \mathcal{Z} \to \R^d$ given by $f(\omega,\zeta)=\zeta_0$, to obtain that
		\[
		\mathbb{E} \left[ \max_{0 \leq s \leq t} \|Z^k_t-Z^k_0\|_2^2\right] \leq \frac{25t}{2} \limsup_{\eps \downarrow 0} \frac{1}{\eps} \mathbb{E} \left[ \|Z^k_\eps-Z^k_0\|_2^2\right]
		\]
		for every $t>0$ and $k\geq 1$. Since the Euclidean displacement is trivially bounded by the total number of jumps, we obtain that
		\[
		\mathbb{E} \left[ \max_{0 \leq s \leq t} \|Z^k_t-Z^k_0\|_2^2\right] \leq \frac{25t}{2} \limsup_{\eps \downarrow 0} \frac{1}{\eps} \mathbb{E} \left[ N_k[0,\eps]^2\right]
		\leq \frac{25t}{2} \limsup_{\eps \downarrow 0} \frac{1}{\eps} \left( \eps \|\eta\|_1 + 2 \eps^2 \|\eta\|_2^2 \right) = \frac{25t}{2} \|\eta\|_1.
		\]
		Taking the limit as $k\to\infty$, it follows by a similar weak convergence and Fatou argument to that used in the proof of \ref{prop:ticks} that
		\[
		\mathbb{E} \left[ \max_{0 \leq s \leq t} \|X_t-X_0\|_2^2\right] \leq \frac{25t}{2} \|\eta\|_1
		\]
		for every $t\geq 0$ also. The claimed two-sided version of this inequality follows by reversibility.
	\end{proof}

\section{Proof of the main theorem}
In this section will will prove Theorem \ref{thm:main} and its corollaries. We begin with the following general criterion for infinite collisions at integer times, from which our main theorems will be deduced. Recall that we write $\mathbb{E}^\eta$ for conditional expectations given the environment $\eta$.
\begin{proposition} \label{prop:EToP}
	Let $d\geq 1$, let $\eta: \R \times E_d \to [0,\infty)$ be a stationary, non-explosive random environment on $\Z^d$ and let $(X_t)_{t\in \R}$ and $(Y_t)_{t\in \R}$ be random walks in $\eta$, both started at the origin at time zero, that are conditionally independent given $\eta$. Then we have the implication
	\begin{equation} \label{eq:hyp} \biggl(\mathbb{E}^\eta\sum_{n\geq0}\ind{X_{-n}=Y_{-n}}=\infty \text{ almost surely}
	\biggr) \Rightarrow
	 \biggl(\,\sum_{n\geq0}\ind{X_n=Y_n}=\infty \text{ almost surely}
	 \biggr).
	\end{equation} 
\end{proposition}

The proof of this proposition is adapted from the methods of \cite{hutchcroft2015}, and relies on the \emph{mass-transport principle} for $\Z^d$.
Recall that a function $f:\Z^d\times\Z^d\rightarrow[0,\infty]$ is said to be a \emph{transport function} if it is diagonally invariant in the sense that $f(x,y)=f(x+z,y+z)$ for every $x,y,z\in\Z^d$. The mass-transport principle for $\Z^d$ states that 
\[
\sum_{x\in\Z^d} f(0,x)=\sum_{x\in\Z^d} f(x,0).
\]
for every transport function $f$.

\begin{proof}
Suppose that $\mathbb{E}^\eta{\sum_{n\geq0}\ind{X_{-n}=Y_{-n}}}=\infty$ almost surely.
Recall that $P_{t_1,t_2}^\eta(\cdot,\cdot)$ denotes the transition probabilities of the random walk conditional on the environment $\eta$.  
For each $u\in\Z^d$ and $n\in \Z$ we let $q_\mathrm{fin}^\eta(u,n)$ denote the conditional probability given $\eta$ that two conditionally independent random walks started at the space-time location $(u,n)$ occupy the same position for only finitely many positive integer times $m\geq n$,
and let $q_0^\eta(u,n)$
denote the conditional probability that the two walks started at $(u,n)$ do not occupy the same position at any integer time strictly greater than $n$.
%
Decomposing according to the last integer time at which the two walks occupy the same position, and where they do so, we get that
\[q_\mathrm{fin}^\eta(u,n) = \sum_{v\in \Z^d}{\sum_{m\geq n}{P_{n,m}^\eta(u,v)^2 q_0^\eta(v,m)}} .\]
By space-shift invariance,  $f(u,v)= \sum_{m\geq 0}\Ex{P_{0,m}^\eta(u,v)^2 q_0^\eta(v,m)}$ is a transport function  and we can apply the mass-transport principle to get that
\[
\E{q^\eta_\mathrm{fin}(0,0)}=\E{ \sum_{v\in \Z^d} \sum_{m\geq 0} {P_{0,m}^{\eta}(0,v)^2 q_0^{\eta}(v,m) }} = 
\E{ \sum_{v\in \Z^d} \sum_{m\geq 0} {P_{0,m}^{\eta}(v,0)^2 q_0^{\eta}(0,m) }},
\]
and hence by time-shift invariance applied to each term that
\begin{align}
\E{q^\eta_\mathrm{fin}(0,0)} &= \E{ \sum_{v\in \Z^d} \sum_{m\geq 0} {P_{-m,0}^{\eta}(v,0)^2 q_0^{\eta}(0,0) }} 
\nonumber
\\&= 
\E{ q_0^{\eta}(0,0) \sum_{v\in \Z^d} \sum_{m\geq 0} {P_{0,-m}^{\eta}(0,v)^2  }} = \E{q_0^\eta(0,0)\Em{\eta}{\sum_{n\geq0}\ind{X_{-n}=Y_{-n}}}}.
\end{align}
Since $q_\mathrm{fin}^\eta(0,0)$ is at most one and
$\mathbb{E}^\eta{\sum_{n\geq0}\ind{X_{-n}=Y_{-n}}}=\infty$ a.s.\ by assumption, we must have that $q_0^\eta(0,0)=0$ a.s.\ and hence that $q_\mathrm{fin}^\eta(0,0)=0$ a.s.\ also. This implies the claim.
\end{proof}

Next, we note that infinite collisions at infinite times quite generally implies that the Lebesgue measure of the set of all positive collision times is infinite almost surely.

\begin{lemma}
\label{lem:discrete_to_continuous_collisions}
	Let $d\geq 1$, let $\eta: \R \times E_d \to [0,\infty)$ be a stationary, non-explosive random environment on $\Z^d$ and let $(X_t)_{t\in \R}$ and $(Y_t)_{t\in \R}$ be random walks in $\eta$, started at $x$ and $y$ at time zero, that are conditionally independent given $\eta$. If the set $\{n \in \N : X_n=Y_n\}$ has infinite cardinality almost surely, then the set $\{t \in [0,\infty) : X_t = Y_t\}$ has infinite Lebesgue measure almost surely.
\end{lemma}

\begin{proof}
	Let $U_1$ and $U_2$ be two conditionally independent Poisson processes with intensity $\eta$ and let $X^s=F_{0,s}(U_1)$ and $Y^s=F_{0,s}(U_2)$ for each $s \in \R$. It follows by stationarity of $\eta$ that the law of $(X^s,Y^s)$ does not depend on $s$.
	Let $T$ be the infimal positive time at which either of the walks $X^0$ or $Y^0$ takes a jump, so that $0<T \leq \infty$ almost surely and $(X^s,Y^s)=(X^0,Y^0)$ for all $0\leq s<T$. Then we have that
	\begin{align*}
	\mathfrak{Leb}\{t\in[0,\infty) :  X^0_t=Y^0_t\}&=\int_0^1 \abs{\{n \in \N :   X^0_{n+s}=Y^0_{n+s}\}} \dif s  \ 
	\\&\geq\int_0^{T\wedge 1} \abs{\{n \in \N :  X^0_{n+s}=Y^0_{n+s}\}} \dif s = 
	\int_0^{T\wedge 1} \abs{\{n \in \N :  X^s_{n+s}=Y^s_{n+s}\}} \dif s.
	\end{align*}
	Since $T>0$ almost surely and the integrand $\abs{\{n \in \N :  X^s_{n+s}=Y^s_{n+s}\}}$ is almost surely infinite for each $s \geq 0$, it follows by Tonelli's theorem that both sides are almost surely infinite, completing the proof.
\end{proof}

We now apply Proposition~\ref{prop:EToP} and Lemma~\ref{lem:discrete_to_continuous_collisions} to prove Theorems~\ref{thm:main} and \ref{thm:main_diffusive}.

\begin{proof}[Proof of Theorem~\ref{thm:main_diffusive}]
For each $K <\infty$ and $\delta>0$, let $A_{K,\delta} \subseteq \Omega$ be the set of environments $\eta$ such that 
\[
\limsup_{n\to\infty} \min_{0 \leq m \leq n} \mathbb{P}^\eta\left( \|X_{-m}\|_2^2 \leq K n\right) \geq \delta.
\]
By assumption, for every $\eps>0$ there exists $K$ and $\delta$ such that  $\mathbb{P}(\eta\in A_{K,\delta})\geq 1-\eps$. Thus, it suffices to prove that if $K<\infty$ and $\delta>0$ then
$\sum_{m=1}^\infty \mathbb{P}^\eta(X_{-m}=Y_{-m}) =\infty$ for \emph{every} environment $\eta\in A_{K,\delta}$.

 Fix $K<\infty$ and $\delta>0$ and suppose that $\eta\in A_{K,\delta}$ holds. We can recursively define a sequence of positive integer times $n_1,n_2,\ldots$, depending on $\eta$, such that $n_{i+1} \geq 2 (n_i+1)$ for each $i \geq 1$ and
\[\min_{0 \leq m \leq n_i} \mathbb{P}^\eta \left(  \|X_{-m}\|_2^2 \leq K n_i\right) \geq \frac{\delta}{2}\]
for every $i\geq 1$. For each $r\geq 1$, let $\Lambda_r \subseteq \Z^2$ be the set of lattice points with Euclidean norm at most $r$. Then there exists a constant $c$ such that
\[
\mathbb{P}^\eta(X_{-m}=Y_{-m})\geq \sum_{x \in \Lambda_r} P_{0,-m}^\eta(0,x)^2 \geq\frac{1}{|\Lambda_r|} \left[\sum_{x \in \Lambda_r} P_{0,-m}^\eta(0,x)\right]^2 \geq \frac{c}{r^2} \mathbb{P}^\eta(X_{-m} \in \Lambda_r)^2
\]
for every $m,r\geq 1$ and hence that
\begin{multline*}
\sum_{m=n_i+1}^{n_{i+1}} \mathbb{P}^\eta(X_{-m}=Y_{-m})\geq 
\frac{c}{K n_{i+1}} \sum_{m =n_i+1}^{n_{i+1}}\mathbb{P}^\eta(\|X_{-m}\|_2^2 \leq K n_{i+1})^2 \\\geq \frac{c}{2 K} \min_{1\leq m \leq n_{i+1}}\mathbb{P}^\eta\left( \|X_{-m}\|_2^2 \leq K n_{i+1}\right)^2 \geq \frac{c \delta^2}{8 K}
\end{multline*}
for every $i\geq 1$. Summing over $i\geq 1$, it follows that $\sum_{m=1}^\infty \mathbb{P}^\eta(X_{-m}=Y_{-m}) =\infty$ as claimed.
\end{proof}

\begin{proof}[Proof of Theorem~\ref{thm:main}]
It suffices to prove that the conditions (A1) and (A2) each imply the weak diffusive estimate on the backwards process \eqref{eq:diffusive_assumption} needed to apply Theorem~\ref{thm:main_diffusive}. This is obvious in the case (A2) that the backwards process satisfies a (quenched or annealed) invariance principle with Brownian scaling. (It is not a problem if the limiting covariance is random.) 
In the case (A1) that the environment is strongly reversible and Markovian, we have by Markov's inequality and Corollary~\ref{prop:mTDist} that
\begin{multline*}
\mathbb{P}\left( \min_{m\leq n} \mathbb{P}^\eta\left(\|X_{-m}\|^2_2 \leq K n\right) \leq \delta \right) \leq 
\mathbb{P}\left( \mathbb{P}^\eta\left(\max_{m\leq n}  \|X_{-m}\|^2_2 > K n\right) \geq 1-\delta \right)
\\\leq
\mathbb{P}\left(\mathbb{E}^\eta\left[\max_{m\leq n} \|X_{-m}\|^2_2\right]
\geq K (1-\delta) n\right)
\leq \frac{25}{K(1-\delta)}\norm{\eta}_1
\end{multline*}
for every $K<\infty$, $\delta>0$, and $n\geq 1$, and hence by Fatou's lemma that
\[
\mathbb{P}\left( \limsup_{n\to\infty}\min_{m\leq n} \mathbb{P}^\eta\left(\|X_{-m}\|^2_2 \leq K n\right) \leq \delta \right) \leq \frac{25}{K(1-\delta)}\norm{\eta}_1
\]
for every $K<\infty$ and $\delta>0$. This implies the claim.
\end{proof}

We next prove Proposition \ref{prop:1dim}, which concerns the one-dimensional case.

\begin{proof}[Proof of Proposition \ref{prop:1dim}]
	Bounding the total displacement by the number of jumps, Proposition~\ref{prop:ticks} 
	implies that $\mathbb{E}\max_{0\leq m \leq n}\|X_{-m}\| \leq \mathbb{E}N[-n,0] \leq n \|\eta_1\|$ for every $n\geq 1$. In the one dimensional case, this linear bound is sufficient to guarantee that $\mathbb{E}^\eta \sum_{n\geq 0} \mathbbm{1}(X_{-n}=Y_{-n})=0$ almost surely; the details are very similar to the proof of Theorem~\ref{thm:main_diffusive} and are omitted.
\end{proof}

It remains only to prove Corollary~\ref{cor:startloc}, which concerns the case that the two walks do not start at the same vertex, and will be deduced from Theorems \ref{thm:main} and \ref{thm:main_diffusive} together with the following general lemma.
\begin{lemma}
\label{lem:changed_start}
	Let $d\geq 1$ and let $\eta: \R \times E_d \to [0,\infty)$ be an irreducible, time-ergodic, stationary random environment on $\Z^d$. Let $(X_t)_{t\in \R}$, $(X_t')_{t\in \R}$, $(Y_t)_{t\in \R}$, and $(Z_t)_{t\in \R}$ be random walks in $\eta$, started at some vertices $x$, $x$, $y$, and $z$ at time zero respectively, that are conditionally independent given $\eta$. If $\{n\in \N: X_n = X_n'\}$ is infinite almost surely, then $\{n\in \N: Y_n = Z_n\}$ is infinite almost surely.
\end{lemma}

\begin{proof}[Proof of Lemma~\ref{lem:changed_start}]
	By stationarity, we can without loss of generality assume that $x=y=0$. For each $z\in\Z^d$ and $t\in\R$ we define $A_{z,t}$ to be the set of environments $\eta$ for which $P^\eta_t(0,z) > 0$.  We will first use irreducibility and time-ergodicity of $\eta$ to prove that $\pr(A_{z,t})\rightarrow 1$ as $t\rightarrow\infty$ for each fixed $z\in \Z^d$. 
	 %
	Irreducibility give us that there exists some $t_0>0$ such that $\eta \in A_{z,t_0}$ with positive probability.
	We deduce by stationarity and time-ergodicity that $\tau_n \eta \in A_{z,t_0}$ for infinitely many positive integers almost surely, and hence that $\mathbb{P}($there exists $m \leq t$ such that $\tau_m \eta \in A_{z,t_0})\to 1$ as $t\to\infty$. Since the walk always has a positive conditional probability not to move in any given time interval, we have that
	\[
	\tau_t \eta \in A_{z,t_0} \iff P^\eta_{t,t+t_0}(0,z)>0 \Rightarrow P^\eta_{0,t+t_0}(0,z)>0 \iff \eta \in A_{z,t+t_0}
	\]
	for every $t\geq 0$, and hence that
	\[
	 \mathbb{P}(\eta \in A_{z,t+t_0}) \geq \mathbb{P}(\text{there exists $m \leq t$ such that $\tau_m \eta \in A_{z,t_0}$})\to 1
	\]
	as $n\to\infty$ as claimed.

	For each $n\in \N$ and $\eta\in A_{z,n},$ the event $B_{u,n}=\{X_n=0, X_n'=z\}$ has positive conditional probability. Let $Y'$ and $Z'$ be random walks on $\eta$, started at $(0,n)$ and $(z,n)$, that are conditionally independent of each other and of $(X,X')$ given $\eta$, so that $(\tau_n Y', \tau_n Z')$ has the same marginal distribution as $(Y,Z)$. We have by the Markov property that
%
	\[
	\pr^{\eta}\left(\sum_{m\geq0}\ind{Y_m'=Z_m'}=\infty\right) = \pr^\eta\left(\sum_{m\geq n}\ind{X_m=Y_m}=\infty\big\vert B_{u,n}\right) = 1
	\] 
	almost surely on the event $A_{z,n}$, and hence by stationarity that
	\[
	\pr\left(\sum_{m\geq0}\ind{Y_m=Z_m}=\infty\right) = \pr\left(\sum_{m\geq 0}\ind{Y_m'=Z_m'}=\infty\right) \geq \pr({A_{z,n}})
	\]
	for every $n\geq 1$. The claim follows since the right hand side tends to $1$ as $n\to\infty$. \qedhere
	
\end{proof}
\subsection*{Acknowledgements}
We thank Sebastian Andres and Jonathan Hermon for helpful comments on a draft of the paper.

\setstretch{1}
\footnotesize{
  \bibliographystyle{abbrv}
  \bibliography{bibliography}

\begin{thebibliography}{10}

\bibitem{aldous1995reversible}
D.~Aldous and J.~Fill.
\newblock Reversible {M}arkov chains and random walks on graphs.

\bibitem{andres2014}
S.~Andres.
\newblock Invariance principle for the random conductance model with dynamic
  bounded conductances.
\newblock {\em Ann. Inst. H. Poincaré Probab. Statist.}, 50:352--374, 05 2014.

\bibitem{andres2018}
S.~Andres, A.~Chiarini, J.-D. Deuschel, and M.~Slowik.
\newblock Quenched invariance principle for random walks with time-dependent
  ergodic degenerate weights.
\newblock {\em Ann. Probab.}, 46(1):302--336, 01 2018.

\bibitem{ACS}
S.~Andres, A.~Chiarini, and M.~Slowik.
\newblock Quenched local limit theorem for random walks among time-dependent
  ergodic degenerate weights.
\newblock 01 2020.

\bibitem{andres2020green}
S.~Andres, J.-D. Deuschel, and M.~Slowik.
\newblock Green kernel asymptotics for two-dimensional random walks under
  random conductances.
\newblock {\em Electronic Communications in Probability}, 25, 2020.

\bibitem{andres2019local}
S.~Andres and P.~A. Taylor.
\newblock Local limit theorems for the random conductance model and
  applications to the {G}inzburg-{L}andau $\nabla\varphi$ interface model.
\newblock {\em arXiv preprint arXiv:1907.05311}, 2019.

\bibitem{avena2012}
L.~Avena.
\newblock Symmetric exclusion as a model of non-elliptic dynamical random
  conductances.
\newblock {\em Electron. Commun. Probab.}, 17:8 pp., 2012.

\bibitem{avena2016class}
L.~Avena, O.~Blondel, and A.~Faggionato.
\newblock A class of random walks in reversible dynamic environments:
  antisymmetry and applications to the east model.
\newblock {\em Journal of Statistical Physics}, 165(1):1--23, 2016.

\bibitem{AVENA20183490}
L.~Avena, O.~Blondel, and A.~Faggionato.
\newblock Analysis of random walks in dynamic random environments via
  {$L^2$}-perturbations.
\newblock {\em Stochastic Processes and their Applications}, 128(10):3490 --
  3530, 2018.

\bibitem{MR1159828}
K.~Ball.
\newblock Markov chains, {R}iesz transforms and {L}ipschitz maps.
\newblock {\em Geom. Funct. Anal.}, 2(2):137--172, 1992.

\bibitem{barlow2012}
M.~T. Barlow, Y.~Peres, and P.~Sousi.
\newblock Collisions of random walks.
\newblock {\em Ann. Inst. H. Poincaré Probab. Statist.}, 48(4):922--946, 11
  2012.

\bibitem{biskup2011recent}
M.~Biskup.
\newblock Recent progress on the random conductance model.
\newblock {\em Probability Surveys}, 8:294--373, 2011.

\bibitem{BISKUP2018985}
M.~Biskup and P.-F. Rodriguez.
\newblock Limit theory for random walks in degenerate time-dependent random
  environments.
\newblock {\em Journal of Functional Analysis}, 274(4):985 -- 1046, 2018.

\bibitem{chen2016gaussian}
X.~Chen.
\newblock Gaussian bounds and collisions of variable speed random walks on
  lattices with power law conductances.
\newblock {\em Stochastic Processes and their Applications},
  126(10):3041--3064, 2016.

\bibitem{CCC}
X.~Chen and D.~Chen.
\newblock Two random walks on the open cluster of {$\mathbb{Z}^2$} meet
  infinitely often.
\newblock {\em Science China Mathematics}, 53(8):1971--1978, 2010.

\bibitem{chen2011}
X.~Chen and D.~Chen.
\newblock Some sufficient conditions for infinite collisions of simple random
  walks on a wedge comb.
\newblock {\em Electron. J. Probab.}, 16:1341--1355, 2011.

\bibitem{SDE}
T.~Delmotte and J.-D. Deuschel.
\newblock On estimating the derivatives of symmetric diffusions in stationary
  random environment, with applications to ${\Delta}\phi$ interface model.
\newblock {\em Probability Theory and Related Fields}, 133:358--390, 11 2005.

\bibitem{devulder2018collisions}
A.~Devulder, N.~Gantert, and F.~Pène.
\newblock Collisions of several walkers in recurrent random environments.
\newblock {\em Electronic Journal of Probability}, 23, 2018.

\bibitem{devulder2019}
A.~Devulder, N.~Gantert, and F.~Pène.
\newblock Arbitrary many walkers meet infinitely often in a subballistic random
  environment.
\newblock {\em Electron. J. Probab.}, 24:25 pp., 2019.

\bibitem{RWME}
D.~Dolgopyat, G.~Keller, and C.~Liverani.
\newblock Random walk in {M}arkovian environment.
\newblock {\em The Annals of Probability}, 36(5):1676--1710, 2008.

\bibitem{gallesco_2013}
C.~Gallesco.
\newblock Meeting time of independent random walks in random environment.
\newblock {\em ESAIM: Probability and Statistics}, 17:257–292, 2013.

\bibitem{ganguly2020chemical}
S.~Ganguly and J.~R. Lee.
\newblock Chemical subdiffusivity of critical 2d percolation.
\newblock {\em arXiv preprint arXiv:2005.08934}, 2020.

\bibitem{gantert2014recurrence}
N.~Gantert, M.~Kochler, and F.~Pene.
\newblock On the recurrence of some random walks in random environment.
\newblock {\em arXiv preprint arXiv:1404.3874}, 2014.

\bibitem{gwynne2020anomalous}
E.~Gwynne and T.~Hutchcroft.
\newblock Anomalous diffusion of random walk on random planar maps.
\newblock {\em Probability Theory and Related Fields}, pages 1--45, 2020.

\bibitem{Helffer1994OnTC}
B.~Helffer and J.~Sjoestrand.
\newblock On the correlation for {K}ac-like models in the convex case.
\newblock {\em Journal of Statistical Physics}, 74:349--409, 1994.

\bibitem{Hermon2020ACP}
J.~Hermon and P.~Sousi.
\newblock A comparison principle for random walk on dynamical percolation.
\newblock {\em arXiv: Probability}, 2020.

\bibitem{hutchcroft2015}
T.~Hutchcroft and Y.~Peres.
\newblock Collisions of random walks in reversible random graphs.
\newblock {\em Electron. Commun. Probab.}, 20:6 pp., 2015.

\bibitem{krishnapur2004}
M.~Krishnapur and Y.~Peres.
\newblock Recurrent graphs where two independent random walks collide finitely
  often.
\newblock {\em Electron. Commun. Probab.}, 9:72--81, 2004.

\bibitem{lee2020relations}
J.~R. Lee.
\newblock Relations between scaling exponents in unimodular random graphs.
\newblock {\em arXiv preprint arXiv:2007.06548}, 2020.

\bibitem{liggett2012interacting}
T.~M. Liggett.
\newblock {\em Interacting particle systems}, volume 276.
\newblock Springer Science \& Business Media, 2012.

\bibitem{MR3616205}
R.~Lyons and Y.~Peres.
\newblock {\em Probability on trees and networks}, volume~42 of {\em Cambridge
  Series in Statistical and Probabilistic Mathematics}.
\newblock Cambridge University Press, New York, 2016.

\bibitem{NashDyn}
J.-C. Mourrat and F.~Otto.
\newblock Anchored {N}ash inequalities and heat kernel bounds for static and
  dynamic degenerate environments.
\newblock {\em Journal of Functional Analysis}, 270, 03 2015.

\bibitem{MR2239346}
A.~Naor, Y.~Peres, O.~Schramm, and S.~Sheffield.
\newblock Markov chains in smooth {B}anach spaces and {G}romov-hyperbolic
  metric spaces.
\newblock {\em Duke Math. J.}, 134(1):165--197, 2006.

\bibitem{Peres2017MixingTF}
Y.~Peres, P.~Sousi, and J.~E. Steif.
\newblock Mixing time for random walk on supercritical dynamical percolation.
\newblock {\em Probability Theory and Related Fields}, 176:809--849, 2017.

\bibitem{peres2017quenched}
Y.~Peres, P.~Sousi, and J.~E. Steif.
\newblock Quenched exit times for random walk on dynamical percolation.
\newblock {\em arXiv preprint arXiv:1707.07619}, 2017.

\bibitem{Peres2015}
Y.~Peres, A.~Stauffer, and J.~E. Steif.
\newblock Random walks on dynamical percolation: mixing times, mean squared
  displacement and hitting times.
\newblock {\em Probability Theory and Related Fields}, 162(3):487--530, Aug
  2015.

\bibitem{MR4073931}
Y.~Peres and T.~Zheng.
\newblock On groups, slow heat kernel decay yields {L}iouville property and
  sharp entropy bounds.
\newblock {\em Int. Math. Res. Not. IMRN}, (3):722--750, 2020.

\bibitem{Redig2018SymmetricSE}
F.~Redig, E.~Saada, and F.~S. Sau.
\newblock Symmetric simple exclusion process in dynamic environment:
  hydrodynamics.
\newblock {\em arXiv: Probability}, 2018.

\bibitem{SPITZER1970246}
F.~Spitzer.
\newblock Interaction of {M}arkov processes.
\newblock {\em Advances in Mathematics}, 5(2):246 -- 290, 1970.

\end{thebibliography}
}

\medskip

\noindent \sc{N.\ Halberstam: CCIMI, University of Cambridge,} \email{nh448@cam.ac.uk}\\
\noindent \sc{T.\ Hutchcroft: Statslab, DPMMS, University of Cambridge,} \email{t.hutchcroft@maths.cam.ac.uk} 

\end{document}